\documentclass[11pt,twoside]{article}
\usepackage{amsfonts,amsmath,amssymb,amsthm}
\usepackage{epsfig,color}
\usepackage{epsfig}
\usepackage{tabularx}
\usepackage{graphicx}
\usepackage{amsthm}
\usepackage{amsmath}
\usepackage{amsfonts}
\usepackage{euscript}
\usepackage{color}
\usepackage{epstopdf}
\usepackage{comment}

\newtheorem{thm}{Theorem}
\newtheorem{lemma}{Lemma}

\newtheorem{define}{Definition}
\newcommand{\OO}{\mathcal{O}}

\theoremstyle{remark}
\newtheorem{rem}{Remark}
\def\l{\left}
\def\r{\right}
\def\R#1{$(\ref{#1})$}
\newcommand{\bb}[1]{\begin{equation}\label{#1}}
\newcommand{\ee}{\end{equation}}
\newcommand{\bbb}{\begin{eqnarray}}
\newcommand{\eee}{\end{eqnarray}}
\newcommand{\bbbb}{\begin{eqnarray*}}
\newcommand{\eeee}{\end{eqnarray*}}
\newcommand{\nnn}{\nonumber}
\everymath{\displaystyle}
\newcommand{\x}{\textbf{x}}

\newcommand{\uu}{\textbf{u}}
\newcommand{\vv}{\textbf{v}}
\newcommand{\f}{\textbf{f}}
\newcommand{\z}{\textbf{z}}

\newcommand{\RR}{\mathbb{R}}

\newcommand{\clearallnum}{
    \numberwithin{equation}{section} \setcounter{equation}{0}
    \numberwithin{thm}{section} \setcounter{thm}{0}
    \numberwithin{lemma}{section} \setcounter{lemma}{0}
    \numberwithin{cor}{section} \setcounter{cor}{0}
    \numberwithin{rem}{section} \setcounter{rem}{0}
    \numberwithin{define}{section} \setcounter{define}{0}}

\usepackage[usenames,dvipsnames,svgnames,table]{xcolor}

\allowdisplaybreaks

\begin{document}
\baselineskip=4.8mm

\begin{flushleft}

{\Large\bf A variable nonlinear splitting algorithm for reaction diffusion systems with self- and cross-diffusion}


\vspace{6mm}

{\bf Matthew A. Beauregard\footnote{Principal and corresponding author. Email
address: beauregama@sfasu.edu.  This author was supported in part by internal research grant (No. 150030-26423-150) from Stephen F. Austin State University.}, Joshua L. Padgett\small{$^2$}}

\vspace{4mm}

{\em\small$^1$Department of Mathematics $\&$ Statistics, Stephen F. Austin State University, Nacogdoches, TX, 75962}

{\em\small$^2$Department of Mathematics $\&$ Statistics, Texas Tech University, Lubbock, TX, 79409-1042}

\vspace{18mm}

\parbox[t]{13.8cm}{{\bf Abstract.}\\ \small Self- and cross-diffusion are important nonlinear spatial derivative terms that are included into biological models of predator-prey interactions. Self-diffusion models overcrowding effects, while cross-diffusion incorporates the response of one species in light of the concentration of another.  In this paper, a novel nonlinear operator splitting method is presented that directly incorporates both self- and cross-diffusion into a computational efficient design. The numerical analysis guarantees the accuracy and demonstrates appropriate criteria for stability.  Numerical experiments display its efficiency and accuracy.}

\vspace{8mm}

\parbox[t]{13.8cm}{\em Keywords: \small reaction-diffusion equations, nonlinear operator splitting, self-diffusion, cross-diffusion}

\end{flushleft}

\section{Introduction} \clearallnum

\linespread{1.0}

This paper is motivated to develop efficient and accurate numerical approximations to a generalized Shigesada-Kawasaki-Teramoto (SKT) model, which is a two-species predator-prey model first developed in \cite{Shig1979} and has been subject to much research ever since. Here, we consider the following self- and cross-diffusion system of equations that models the interaction between the prey ($u$) and predator ($v$),
\begin{eqnarray}
\label{eq:1.1}
u_t - \Delta\left(d_1 u + s_1 u^2 + c_{12} v u\right) &=& f(u,v) \\
\label{eq:1.2}
v_t - \Delta\left(d_2 v + s_2 v^2 + c_{21} u v\right) &=& g(u,v) 
\end{eqnarray}
\noindent defined on $\mathbb{R}^{+}\times \Omega$. Here $\Omega = [0,L]\times [0,L]$, $\Delta$ is the two-dimensional Laplacian operator, and the reactive functions $f$ and $g$ are in $C^1(\mathbb{R}^2)$. We define $\x$ to be the spatial coordinate vector in two dimensions. The initial populations are given as
\[ u(0,\x)=u_{0}(\x),~~ v(0,\x)=v_{0}(\x)~~~~~ \x \in \Omega, \]
\noindent and are assumed to be nonnegative and in $C^1(\Omega)$. The diffusion coefficients are given by $d_i > 0$. The parameters $c_{12}$ and $c_{21}$ describe the cross-diffusion response, with the parameter $c_{12}$ being nonnegative. This parameter choice encourages a prey to move away from high concentrations of the predator. A similar description is given for $c_{21} > 0$. In the event that $c_{21}<0$ the species $v$ move toward higher concentrations of $u$, its primary food source.  However, in this case, finite time blow-up of the solution may occur \cite{Zhu2014}, hence, we assume the $c_{21}>0$.  The nonnegative self-diffusion coefficients $s_i$ model interspecies competition and provides a severe penalty to over-crowding of a particular species.

Since each of the coefficients are assumed nonnegative then it is known that weak solutions exist globally in time when considering homogenous Neumann boundary conditions and reactive functions of Lotka-Volterra type that satisfy $f(0,v)=g(u,0)=0$ \cite{Chen2018,Chen2004,Chen2006}. Likewise, classical solutions globally exist under various restrictions on the self and cross diffusion coefficients \cite{Le2006}, namely, $s_1s_2 \geq 0, ~ s_2 > c_{12}, ~ s_1 > c_{21}.$  However, if $d_1<d_2$ then $s_1$ maybe less than $c_{21}$ and solutions will still exist globally in time \cite{Kouachi2014}.  In light of these theoretical results, this paper only considers homogeneous Neumann boundary conditions, even though our numerical analysis extends to homogeneous Dirichlet or Neumann-Dirichlet boundary conditions in a straightforward manner.

The inclusion of self- and cross-diffusion terms allow for realistic responses to predator and prey movement and are often incorporated into mathematical models in population biology \cite{Berres2011,Gierer1972,Yong2012,Yong2016}.  Our current efforts consider only a two species model, yet the methods developed herein can by readily extended to more general systems.
The considered model is rich in dynamics and stems from the generalizations of the SKT model and the literature is abundant with investigations of variants to this model \cite{AA02, Chen2004, Chen2006,Chenli2013,Kirane1996, Kouachi2014, Kuto2009, Shim2006, Xu2008}. However, the development and analysis of accurate and efficient numerical approximations has not been considered.

This paper is organized as follows. In Section 2 we present the nonlinear variable time splitting method where $s_i=0$. The algorithm is based on a variant of a splitting method recently developed by the authors for similar equations with self-diffusion in \cite{BeauPadgett2017}.  Section 3 develops the nonlinear splitting algorithm for systems with both self- and cross-diffusion. In both of these sections the methods are overall second-order accurate and their equivalent factorizations are preferable as they allow for parallel implementations and use of high performance computing (HPC) methods.  Section 4 details further numerical analysis to provide criteria that guarantee stability. Section 5 provides examples to illustrate the efficiency and accuracy of the algorithm.  We conclude and highlight our results in Section 6.

In the following, we define a scheme as \textit{computationally efficient} if it is \textit{second-order accurate} in space and time or better and the \textit{number of operations per time step is directly proportional to the number of unknowns}. In addition, all lowercase bold letters indicate vectors,
uppercase letters are used for matrices. Lastly, the $\ell^2$-norm is used throughout unless otherwise specified.

\section{Splitting algorithm with cross-diffusion only}

Given $N\gg 0,$ we may inscribe over $\Omega$ the mesh $\mathcal{D}_\delta = \{(x_i,y_j)~|~i,j=0,1,\dots,N+1\},$ where $\delta=L/(N+1)$ and $x_i=i\delta$ and $y_i=j\delta$ for $i,j=0,1,\dots,N+1.$ Further,  define $u_{i,j}(t)$ and $v_{i,j}(t)$ as the approximations to the exact solution $u(x_i, y_j, t)$ and $v(x_i, y_j, t)$, respectively.  Let
\begin{eqnarray*}
\textbf{u}&=&(u_{0,0}(t), u_{1,0}(t), \ldots, u_{N+1,0}(t), \ldots, u_{N+1,N+1}(t))^{\top},\\
\textbf{v}&=&(v_{0,0}(t), v_{1,0}(t), \ldots, v_{N+1,0}(t), \ldots, v_{N+1,N+1}(t))^{\top}.
\end{eqnarray*}
A similar notation is used for the reactive functions $f$ and $g$. Second-order central difference approximations are used to approximate the two-dimensional Laplacian operator in Eqs.~(\ref{eq:1.1})-(\ref{eq:1.2}) which yields the semidiscretized equations,
\begin{eqnarray*}
\frac{d \textbf{u}}{dt} &=& \left(d_1 (P + R) + c_{12}( P D(\textbf{v}) + R D(\textbf{v})) \right)\textbf{u} + \textbf{f} \\
\frac{d \textbf{v}}{dt} &=& \left(d_2 (P + R) + c_{21}( P D(\textbf{u}) + R D(\textbf{u})) \right)\textbf{v} + \textbf{g}
\end{eqnarray*}
where $P$, $R$, and $D(\textbf{q})=\mbox{diag}(\textbf{q})$ are $M^2\times M^2$ matrices, where $M=N+2$. The matrices are defined as $P\mathrel{\mathop:}=I_{M}\otimes T$ and $R\mathrel{\mathop:}=T\otimes I_{M},$ where $I_{M}$ is the $M\times M$ identity matrix and $T$ is a $M\times M$ tridiagonal matrix with main diagonal $-2/\delta^2$ and upper and lower diagonals of $1/\delta^2$ with the exception of $T_{12}=T_{M,M-1} = 2$ as a result of the Neumann boundary conditions. The semidiscretized equations are specified at every mesh point, with the Neumann boundary conditions being incorporated via central difference approximations.

A variable time-step second-order Crank-Nicolson method \cite{BeauPadgett2017} is used to advance the solution in time, that is,
{\small
\begin{eqnarray}
& &\left(I - \frac{\tau_k}{2} \left((P + R)(d_1 + c_{12} D_{k+1}^{(\textbf{v})} ) \right)\right)\textbf{u}_{k+1}\nnn\\
&&~~~~~~~ = \left(I + \frac{\tau_k}{2} \left((P + R)(d_1 + c_{12}D_k^{(\textbf{v})}) \right)\right)\textbf{u}_{k} + \frac{\tau_k}{2}\left(\textbf{f}_{k+1} + \textbf{f}_k\right)\label{crank1} \\
&&\left(I - \frac{\tau_k}{2} \left((P + R)(d_2 + c_{21}D_{k+1}^{(\textbf{u})}) \right)\right)\textbf{v}_{k+1}\nnn\\
&&~~~~~~~ = \left(I + \frac{\tau_k}{2} \left((P + R)(d_2 + c_{21} D_k^{(\textbf{u})}) \right)\right)\textbf{v}_{k} + \frac{\tau_k}{2}\left(\textbf{g}_{k+1} + \textbf{g}_k\right)\label{crank2}
\end{eqnarray}}
where $\textbf{u}_k$ and $\textbf{v}_k$ are approximations to $\textbf{u}$ and $\textbf{v}$ at time $t_k = \textstyle\sum_{i=0}^{k-1} \tau_k$ and $D_{k}^{(\textbf{q})} = D(\textbf{q}_k)$ are diagonal matrices whose elements are $\textbf{q} = \textbf{u},\,\textbf{v}.$ Similarly, the reactive functions $\textbf{f}_k$ and $\textbf{g}_k$ denotes the estimated values of $\textbf{f}$ and $\textbf{g}$ using the approximations of $\textbf{u}_k$ and $\textbf{v}_k$.  To advance the solution in time we propose to solve Eqs.~(\ref{crank1})-(\ref{crank2}) through a modified Douglass-Gunn splitting method, similar to that shown in \cite{BeauPadgett2017}, given by
{\small
\begin{subequations}
\begin{eqnarray}
\label{SplitStepFirst}
   \l(I - \frac{d_1\tau_k}{2} P\r)\tilde{\textbf{u}}^{(1)} &=& \left(I + \frac{\tau_k}{2} \left(d_1(P + 2R) + 2c_{12}(P+R)D_k^{(\textbf{v})}\right)\right) \textbf{u}_k + \tau_k \textbf{f}_k,~~~~~~~~ \\
   \l(I - \frac{d_2\tau_k}{2} P\r)\tilde{\textbf{v}}^{(1)} &=& \left(I + \frac{\tau_k}{2} \left(d_2(P + 2R) + 2c_{21}(P+R)D_k^{(\textbf{u})}\right)\right) \textbf{v}_k + \tau_k \textbf{g}_k,~~~~ \\
 \l(I - \frac{d_1\tau_k}{2} R\r)\tilde{\textbf{u}}^{(2)} &=& \tilde{\textbf{u}}^{(1)} - \frac{d_1\tau_k}{2} R \textbf{u}_k,\\
 \l(I - \frac{d_2\tau_k}{2} R\r)\tilde{\textbf{v}}^{(2)} &=& \tilde{\textbf{v}}^{(1)} - \frac{d_2\tau_k}{2} R \textbf{v}_k,\\
 \l(I - \frac{c_{12}\tau_k}{2} PD_{k+1}^{(\textbf{v})}\r)\tilde{\textbf{u}}^{(3)} &=& \tilde{\textbf{u}}^{(2)} - \frac{c_{12}\tau_k}{2} P D_k^{(\textbf{v})} \textbf{u}_k,\label{splitstepa}\\
 \l(I - \frac{c_{21}\tau_k}{2} PD_{k+1}^{(\textbf{u})}\r)\tilde{\textbf{v}}^{(3)} &=& \tilde{\textbf{v}}^{(2)} + \frac{c_{21}\tau_k}{2} P D_k^{(\textbf{u})} \textbf{v}_k,\\
 \l(I - \frac{c_{12}\tau_k}{2} RD_{k+1}^{(\textbf{v})}\r)\textbf{u}_{k+1} &=& \tilde{\textbf{u}}^{(3)} - \frac{c_{12}\tau_k}{2} R D_k^{(\textbf{v})} \textbf{u}_k + \frac{\tau_k}{2}\left( \textbf{f}_{k+1} - \textbf{f}_k\right),\\
\label{SplitStepFinal}
 \l(I - \frac{c_{21}\tau_k}{2} RD_{k+1}^{(\textbf{u})}\r)\textbf{v}_{k+1} &=& \tilde{\textbf{v}}^{(3)} - \frac{c_{21}\tau_k}{2} R D_k^{(\textbf{u})} \textbf{v}_k + \frac{\tau_k}{2}\left( \textbf{g}_{k+1} - \textbf{g}_k\right).
\end{eqnarray}
\end{subequations}}

Note that \R{splitstepa}-\R{SplitStepFinal} involve implicit terms. In order to maintain the desired computational efficiency, the implicit terms are approximated by taking an Euler step, that is,
\begin{eqnarray*}
\textbf{u}_{k+1} &=& \textbf{u}_k + \tau_k\left( (P+R)(d_1 + c_{12}D_k^{(\textbf{v})})\right) \textbf{u}_k + \tau_k \textbf{f}_k + \OO\l(\tau_k^2\r), \\
\textbf{v}_{k+1} &=& \textbf{v}_k + \tau_k\left( (P+R)(d_2 + c_{21}D_k^{(\textbf{u})})\right) \textbf{v}_k + \tau_k \textbf{g}_k + \OO\l(\tau_k^2\r).
\end{eqnarray*}
By employing the Euler step, we have developed a second-order accurate scheme which remains computationally efficient using standard ADI techinques \cite{Beauregard2013}. The Euler approximation is established in the first two steps (2.3a) and (2.3b).  Hence, the overall computational cost is still $\OO(N^2)$. In effect, the splitting method resolves the nonlinearity directly and advances the solution by iterating through one-dimensional subproblems. Careful implementation only requires the use of six $N\times 1$ vectors.  The splitting method's solution is equivalent to the solution of Eqs.~(\ref{crank1})-(\ref{crank2}). We demonstrate this equivalence in the following. First, we show that Eqs.~(\ref{crank1})-(\ref{crank2}) are equivalent to a useful factorization.
\begin{thm}
The induced error between the factorizations
\begin{eqnarray}
&& \left(I - \frac{d_1\tau_k}{2}P\right)\left(I - \frac{d_1\tau_k}{2}R\right)\left(I - \frac{c_{12}\tau_k}{2}PD_{k+1}^{(\textbf{v})}\right)\left(I - \frac{c_{12}\tau_k}{2}RD_{k+1}^{(\textbf{v})}\right)\textbf{u}_{k+1} \nonumber \\
\label{ufactoredNOADI}
&& ~~~~=\left(I + \frac{d_1\tau_k}{2}P\right)\left(I + \frac{d_1\tau_k}{2}R\right)\left(I + \frac{c_{12}\tau_k}{2}PD_{k}^{(\textbf{v})}\right)\left(I + \frac{c_{12}\tau_k}{2}RD_{k}^{(\textbf{v})}\right)\textbf{u}_{k}\nonumber\\&& ~~~~~~~ + \frac{\tau_k}{2}(\textbf{f}_{k+1} + \textbf{f}_k),\quad\quad\\
&& \left(I - \frac{d_2\tau_k}{2}P\right)\left(I - \frac{d_2\tau_k}{2}R\right)\left(I - \frac{c_{21}\tau_k}{2}PD_{k+1}^{(\textbf{u})}\right)\left(I - \frac{c_{21}\tau_k}{2}RD_{k+1}^{(\textbf{u})}\right)\textbf{v}_{k+1} \nonumber \\
&& ~~~~=\left(I + \frac{d_2\tau_k}{2}P\right)\left(I + \frac{d_2\tau_k}{2}R\right)\left(I + \frac{c_{21}\tau_k}{2}PD_{k}^{(\textbf{u})}\right)\left(I + \frac{c_{21}\tau_k}{2}RD_{k}^{(\textbf{u})}\right)\textbf{v}_{k} \nonumber \\ &&~~~~~~~ + \frac{\tau_k}{2}(\textbf{g}_{k+1} + \textbf{g}_k),\label{ufactoredNNOADI}
\end{eqnarray}
and Eqs.~(\ref{crank1})-(\ref{crank2}), respectively, is $\OO(\tau_k^3)$.
\end{thm}
\begin{proof}
We show the equivalency between Eq.~(\ref{ufactoredNOADI}) and Eq.~(\ref{crank1}).  A virtually identical proof can be used to establish the result for $\textbf{v}$. For ease of exposition, we disregard the reaction function approximations $\textbf{f}_k$ and $\textbf{f}_{k+1}.$ Expanding the matrix products on the left- and right-hand sides of Eq.~(\ref{ufactoredNOADI}) yields,
\begin{eqnarray*}
&& I - \frac{\tau}{2}\left( d_1(P + R) + c_{12}( P + R )D_{k+1}^{(\textbf{v})}\right) + \frac{\tau^2}{4}\left( d_1^2 PR + \left(c_{12}^2 P D_{k+1}^{(\textbf{v})}R \right.\right.\\
&&~~~~~~~~~~ \left.\left. + d_1c_{12} P^2 + d_1c_{12} PR + d_1c_{12} RP + d_1c_{12} R^2\right)D_{k+1}^{(\textbf{v})}\right) + \OO(\tau_k^3) \\
&& = I + \frac{\tau}{2}\left( d_1(P + R) + c_{12}( P + R)D_{k}^{(\textbf{v})}\right) + \frac{\tau^2}{4}\left( d_1^2 PR + \left(c_{12}^2 P D_{k}^{(\textbf{v})}R \right.\right.\\
&&~~~~~~~~~~ \left.\left. + d_1c_{12} P^2 + d_1c_{12} PR + d_1c_{12} RP + d_1c_{12} R^2\right)D_{k}^{(\textbf{v})}\right) + \OO(\tau_k^3).
\end{eqnarray*}
The desired result is obtained by expanding the $\OO(\tau_k^2)$ terms about $\vv_k$ in the left-hand side: that is $D_{k+1}=D_k + \OO\left(\tau_k\right)$. The truncated terms contribute terms of order $\OO(\tau^3)$. Therefore the $\OO(\tau^2)$ terms on both sides cancel. Consequently the factorization is equivalent up to $\OO(\tau_k^3)$.
\end{proof}
The factorization in Eqs.~\R{ufactoredNOADI} and \R{ufactoredNNOADI} is desirable as it leads to our splitting method given in Eqs.~(2.3a)-(2.3h). This is encapsulated in the following theorem.
\begin{thm}
The induced error between the factorizations given by Eqs.~\R{ufactoredNOADI} and \R{ufactoredNNOADI}
%
and the splitting design of Eqs.~(\ref{SplitStepFirst})-(\ref{SplitStepFinal}) is $\OO(\tau_k^3)$.
\end{thm}
\begin{proof}
This is a straightforward calculation accomplished by multiplying (2.3g) by the matrix product:
$$\left(I - \frac{d_1\tau_k}{2}P\right)\left(I-\frac{d_1\tau_k}{2}R\right)\left(I-\frac{c_{12}\tau_k}{2}PD_{k+1}^{(\textbf{v})}\right).$$
Expanding and substituting into the result the definitions for $\tilde{\textbf{u}}^{(1)}, \tilde{\textbf{u}}^{(2)},$ and $\tilde{\textbf{u}}^{(3)}$ results in
\begin{eqnarray*}
&&\left(I - \frac{d_1\tau_k}{2}P\right)\left(I-\frac{d_1\tau_k}{2}R\right)\left(I-\frac{c_{12}\tau_k}{2}PD_{k+1}^{(\textbf{v})}\right)\left(I-\frac{c_{12} \tau_k}{2}RD_{k+1}^{(\textbf{v})}\right)\textbf{u}_{k+1} \\
&&~~~~~~~~~~= \left(I + \frac{\tau_k}{2}\left( d_1(P + R) + c_{12}(P+R)D_k^{(\textbf{v})}\right)\right)\textbf{u}_k + \frac{\tau_k}{2}(\textbf{f}_{k+1}+\textbf{f}_k) \\
&&~~~~~~~~~~~~~+ \frac{\tau_k^2}{4}\left( (RP + R^2 + P D_k^{(\textbf{v})} R + P^2 + PR)D_k^{(\textbf{v})} + PR  \right)\textbf{u}_k + \OO(\tau_k^3),
\end{eqnarray*}
where we have used the expansions $\vv_{k+1}=\vv_k + \OO\left(\tau_k\right)$ and $\textbf{f}_{k+1} = \textbf{f}_k + \OO(\tau_k)$. The $\textbf{u}_k$ terms on the right-hand side can be factored to yield the desired result. A similar proof can be established for the factorization used to advance $\textbf{v}$.
\end{proof}
The combination of these two theorems establishes that the splitting method is second-order accurate in time. Hence, the overall accuracy of the splitting method combined with the spatial discretization is second-order accurate.

\section{Splitting algorithm with self and cross-diffusion}

Here, we adopt the same notation as in the previous section and consider the semidiscretization of Eqs.~(\ref{eq:1.1})-(\ref{eq:1.2}),
\begin{eqnarray*}
\frac{d \textbf{u}}{dt} &=& \left[ (P + R)(d_1 + s_1 D(\uu) + c_{12} D(\textbf{v})) \right]\textbf{u} + \textbf{f}, \\
\frac{d \textbf{v}}{dt} &=& \left[ (P + R)(d_2 + s_2 D(\vv) + c_{21} D(\textbf{u})) \right]\textbf{v} + \textbf{g}.
\end{eqnarray*}
As before, we advance solution in time using a variable time-step second-order Crank-Nicolson method,
\begin{eqnarray}
& &\left(I - \frac{\tau_k}{2} \left((P + R)(d_1+s_1 D_{k+1}^{(\uu)} + c_{12} D_{k+1}^{(\textbf{v})}) \right)\right)\textbf{u}_{k+1}\nnn\\
&&~~~~~~~~ = \left(I + \frac{\tau_k}{2} \left((P + R)(d_1+s_1 D_{k}^{(\uu)} + c_{12} D_k^{(\textbf{v})}) \right)\right)\textbf{u}_{k} + \frac{\tau_k}{2}\left(\textbf{f}_{k+1} + \textbf{f}_k\right),\quad~~\label{crankSELFCROSS1} \\
&&\left(I - \frac{\tau_k}{2} \left((P + R)(d_2+s_2 D_{k+1}^{(\vv)} + c_{21}D_{k+1}^{(\textbf{u})} ) \right)\right)\textbf{v}_{k+1}\nnn\\
&&~~~~~~~~ = \left(I + \frac{\tau_k}{2} \left((P + R)(d_2+s_2 D_{k}^{(\vv)}+ c_{21} D_k^{(\textbf{u})}) \right)\right)\textbf{v}_{k} + \frac{\tau_k}{2}\left(\textbf{g}_{k+1} + \textbf{g}_k\right).\quad~~\label{crankSELFCROSS2}
\end{eqnarray}
To advance the solution, just as before, a modified Douglass-Gunn splitting method is utilized to solve Eqs.~(\ref{crankSELFCROSS1})-(\ref{crankSELFCROSS2}). The proposed method is given by
{\small
\begin{subequations}
\begin{eqnarray}
\label{SplitStepFirstFull}
   \l(I - \frac{d_1\tau_k}{2} P\r)\tilde{\uu}^{(1)} - \tau_k \textbf{f}_k &=& \left(I + \frac{\tau_k}{2} \left(d_1(P + 2R) \right.\right. \nonumber \\ && ~~~ \left.\left.+ 2(P+R)( s_1 D_k^{(\uu)} + c_{12}D_k^{(\vv)})\right)\right) \uu_k, \quad\\
   \l(I - \frac{d_2\tau_k}{2} P\r)\tilde{\vv}^{(1)} - \tau_k \textbf{g}_k &=& \left(I + \frac{\tau_k}{2} \left(d_2(P + 2R) \right.\right. \nonumber \\ && ~~~ \left.\left.+ 2(P+R)( s_2 D_k^{(\vv)} + c_{21}D_k^{(\uu)})\right)\right) \vv_k, \\
 \l(I - \frac{d_1\tau_k}{2} R\r)\tilde{\uu}^{(2)} &=& \tilde{\uu}^{(1)} - \frac{d_1\tau_k}{2} R \uu_k,\\
 \l(I - \frac{d_2\tau_k}{2} R\r)\tilde{\vv}^{(2)} &=& \tilde{\vv}^{(1)} - \frac{d_2\tau_k}{2} R \vv_k,\\
 \l(I - \frac{s_{1}\tau_k}{2} PD_{k+1}^{(\uu)}\r)\tilde{\uu}^{(3)} &=& \tilde{\uu}^{(2)} - \frac{s_{1}\tau_k}{2} P D_k^{(\uu)} \uu_k,\\
 \l(I - \frac{s_{2}\tau_k}{2} PD_{k+1}^{(\vv)}\r)\tilde{\vv}^{(3)} &=& \tilde{\vv}^{(2)} - \frac{s_{2}\tau_k}{2} P D_k^{(\vv)} \vv_k,\\
 \l(I - \frac{s_{1}\tau_k}{2} RD_{k+1}^{(\uu)}\r)\tilde{\uu}^{(4)} &=& \tilde{\uu}^{(3)} - \frac{s_{1}\tau_k}{2} R D_k^{(\uu)} \uu_k,\\
 \l(I - \frac{s_{2}\tau_k}{2} RD_{k+1}^{(\vv)}\r)\tilde{\vv}^{(4)} &=& \tilde{\vv}^{(3)} - \frac{s_{2}\tau_k}{2} R D_k^{(\vv)} \vv_k,\\
 \l(I - \frac{c_{12}\tau_k}{2} PD_{k+1}^{(\vv)}\r)\tilde{\uu}^{(5)} &=& \tilde{\uu}^{(4)} - \frac{c_{12}\tau_k}{2} P D_k^{(\vv)} \uu_k,\\
 \l(I - \frac{c_{21}\tau_k}{2} PD_{k+1}^{(\uu)}\r)\tilde{\vv}^{(5)} &=& \tilde{\vv}^{(4)} - \frac{c_{21}\tau_k}{2} P D_k^{(\uu)} \vv_k,\\
 \l(I - \frac{c_{12}\tau_k}{2} RD_{k+1}^{(\vv)}\r)\uu_{k+1} &=& \tilde{\uu}^{(5)} - \frac{c_{12}\tau_k}{2} R D_k^{(\vv)} \uu_k + \frac{\tau_k}{2}\left( \textbf{f}_{k+1} - \textbf{f}_k\right),\\
\label{SplitStepFinalFull}
 \l(I - \frac{c_{21}\tau_k}{2} RD_{k+1}^{(\uu)}\r)\vv_{k+1} &=& \tilde{\vv}^{(5)} - \frac{c_{21}\tau_k}{2} R D_k^{(\uu)} \vv_k + \frac{\tau_k}{2}\left( \textbf{g}_{k+1} - \textbf{g}_k\right).
\end{eqnarray}
\end{subequations}}
An Euler step is used to approximate the implicit terms and maintain the computational efficiency.
%
%
The splitting method's solution is equivalent to the solution of Eqs.~(\ref{crankSELFCROSS1})-(\ref{crankSELFCROSS2}).  The following theorems develop an equivalent factorization of the Eqs.~(\ref{crankSELFCROSS1})-(\ref{crankSELFCROSS2}) and show that this factorization is equivalent to the splitting in Eqs.~(\ref{SplitStepFirstFull})-(\ref{SplitStepFinalFull}).  The proofs are omitted since they follow a similar pattern to those of Theorems 2.1 and 2.2.

In the following, for brevity, we employ the notations 
$$M_1(\tau_k,\uu,\vv,d_1,s_1)\mathrel{\mathop:}=\left(I - \frac{d_1\tau_k}{2}P\right)\left(I - \frac{d_1\tau_k}{2}R\right)\left(I - \frac{s_{1}\tau_k}{2}PD_{k+1}^{(\uu)}\right)\left(I - \frac{s_{1}\tau_k}{2}RD_{k+1}^{(\uu)}\right)$$
and
$$M_2(\tau_k,\uu,\vv,d_1,s_1)\mathrel{\mathop:}= \left(I + \frac{d_1\tau_k}{2}P\right)\left(I + \frac{d_1\tau_k}{2}R\right)\left(I + \frac{s_{1}\tau_k}{2}PD_{k}^{(\uu)}\right)\left(I - \frac{s_{1}\tau_k}{2}RD_{k}^{(\uu)}\right).$$
\begin{thm}
The induced error between the factorizations{\small
\begin{eqnarray}
&& M_1(\tau_k,\emph{\uu},\emph{\vv},d_1,s_1)\left(I - \frac{c_{12}\tau_k}{2}PD_{k+1}^{(\emph{\vv})}\right)\left(I - \frac{c_{12}\tau_k}{2}RD_{k+1}^{(\emph{\vv})}\right)\emph{\uu}_{k+1}\nnn\\
&&~~~ = M_2(\tau_k,\emph{\uu},\emph{\vv},d_1,s_1)\left(I + \frac{c_{12}\tau_k}{2}PD_{k}^{(\emph{\vv})}\right)\left(I + \frac{c_{12}\tau_k}{2}RD_{k}^{(\emph{\vv})}\right)\emph{\uu}_{k} + \frac{\tau_k}{2}(\emph{\textbf{f}}_{k+1} + \emph{\textbf{f}}_k),~~~~~~~\label{sc1} \\
&& M_1(\tau_k,\emph{\vv},\emph{\uu},d_2,s_2)\left(I - \frac{c_{21}\tau_k}{2}PD_{k+1}^{(\emph{\uu})}\right)\left(I - \frac{c_{21}\tau_k}{2}RD_{k+1}^{(\emph{\uu})}\right)\emph{\vv}_{k+1} \nonumber \\
&&~~~ = M_2(\tau_k,\emph{\vv},\emph{\uu},d_2,s_2)\left(I + \frac{c_{21}\tau_k}{2}PD_{k}^{(\emph{\uu})}\right)\left(I + \frac{c_{21}\tau_k}{2}RD_{k}^{(\emph{\uu})}\right)\emph{\vv}_{k} + \frac{\tau_k}{2}(\emph{\textbf{g}}_{k+1} + \emph{\textbf{g}}_k)\label{sc2}
\end{eqnarray}}
and Eqs.~(\ref{crankSELFCROSS1})-(\ref{crankSELFCROSS2}), respectively, is $\OO(\tau_k^3)$.
\end{thm}
\begin{thm}
The induced error between the factorizations given by Eqs~\R{sc1}-\R{sc2}
%
and the splitting design of Eqs.~(\ref{SplitStepFirstFull})-(\ref{SplitStepFinalFull}) is $\OO(\tau_k^3)$.
\end{thm}

\section{Stability and Convergence}

We now establish criteria to guarantee the invertibility of specific matrices of the proposed operator splitting method.
\begin{thm} Let $\kappa = \max\{d_1,d_2,s_1,s_2,c_{12},c_{21}\}$. If $\tau_k$ is sufficiently small,
\bb{cfl}
 \frac{\kappa \tau_k}{\delta^2} < \frac{1}{2\max\{1,\max_{j}\{(\emph{\uu}_k)_j,(\emph{\vv}_k)_j\}\}},
\ee
and $\emph{\uu}_k, \emph{\vv}_k \geq 0$ 
, then, for $i=1$ and $2$, the matrices
$$I-\frac{d_i\tau_k}{2}P,~~I-\frac{d_i\tau_k}{2}R,~~I-\frac{s_1\tau_k}{2}PD_{k+1}^{(\emph{\uu})},~~I-\frac{s_1\tau_k}{2}RD_{k+1}^{(\emph{\uu})}, ~~I-\frac{s_2\tau_k}{2}PD_{k+1}^{(\emph{\vv})},~~I-\frac{s_2\tau_k}{2}RD_{k+1}^{(\emph{\vv})},$$
$$I-\frac{c_{12} \tau_k}{2}PD_{k+1}^{(\emph{\vv})},~~I-\frac{c_{12}\tau_k}{2}RD_{k+1}^{(\emph{\vv})}~~I-\frac{c_{21} \tau_k}{2}PD_{k+1}^{(\emph{\uu})},~~I-\frac{c_{21}\tau_k}{2}RD_{k+1}^{(\emph{\uu})},$$
are nonsingular.
\end{thm}
\begin{proof}
In the case of homogeneous Dirichlet or Neumann boundary conditions we have $\|T\|\leq 4/\delta^2.$ Therefore,
\begin{eqnarray*}
\left\| \frac{\tau_k d_1 P}{2} \right\| = \frac{\tau_k d_1}{2} \left\|I_{N+2}\otimes T \right\| = \frac{\tau_k d_1}{2} \| T\| \leq \frac{2 d_1 \tau_k}{\delta^2} < 1,
\end{eqnarray*}
which gives that the matrix $\textstyle I-\frac{d_1\tau_k}{2}P$ is nonsingular.  A similar argument can be established for the matrices $\textstyle I-\frac{d_2\tau_k}{2}P$ and  $\textstyle I-\frac{d_i\tau_k}{2}R$ for $i=1$ and $2$.

For the remaining matrices we substitute Taylor expansions of $\uu_{k+1}$ and $\vv_{k+1}$ into the matrices and then show that the resulting matrices satisfies the weak-row sum criterion \cite{Henrici} (up to the appropriate order), that is:
\begin{enumerate}
\item The nonzero off-diagonal entries share the same sign and are of opposite sign of the nonzero main diagonal entries and;
\item The row sums of the matrix are all nonnegative with at least one row sum strictly greater than zero.
\end{enumerate}

Suppose $c_{21}>0$ and consider the matrix $\textstyle I-\frac{c_{21} \tau_k}{2}PD_{k+1}^{(\uu)}$.  Taylor expanding $\uu_{k+1}$ shows that
\begin{eqnarray*}
I-\frac{c_{21} \tau_k}{2}PD_{k+1}^{(\uu)} = Q + \OO\l(\tau_k^3\r),
\end{eqnarray*}
where
$$ Q = I-\frac{c_{21}\tau_k}{2}PD_{k}^{(\uu)} - \frac{c_{21}\tau_k^2}{2}P\mbox{diag}\l( \left[ (P+R)(d_2+ s_2 D_k^{(\vv)} +  c_{21}D_k^{(\uu)})\right]\uu_k+\f_k\r).$$

Since the splitting scheme is second-order accurate we shall neglect the high-order terms in showing that $Q$ satisfies the weak-row sum criterion. Consider the off diagonal entry:
{\small \bbbb
&&Q_{j,j+1} =-\frac{c_{21}\tau_k}{2\delta^2}\left( (\uu_k)_{j} + \tau_k(\f_k)_{j} + \tau_k \mbox{diag}\left((P+R)(d_1+s_1D_k^{(\uu)} + c_{12}D_k^{(\vv)})\uu_k\right)_{j+1,j+1}\right)\\
&&~~=-\frac{c_{21} \tau_k}{2\delta^2}\left( (\uu_k)_{j} + \tau_k(\f_k)_{j} + \tau_k [(\uu_k)_{j-1}-4(\uu_k)_{j}+(\uu_k)_{j+1}+(\uu_k)_{j+M}+(\uu_k)_{j-M}]\right.\\
&&~ \left. + s_1 \tau_k[(\uu_k)_{j-1}^2-4(\uu_k)_{j}^2+(\uu_k)_{j+1}^2+(\uu_k)_{j+M}^2+(\uu_k)_{j-M}^2]\right.\\
&&~ \left. + c_{12}\tau_k [(\uu_k)_{j-1}(\vv_k)_{j-1}-4(\uu_k)_{j}(\vv_k)_{j}+(\uu_k)_{j+1}(\vv_k)_{j+1}+(\uu_k)_{j+M}(\vv_k)_{j+M}\right. \\
&&~ \left. +(\uu_k)_{j-M}(\vv_k)_{j-M}] \right).
\eeee}
Utilizing Taylor expansions in the spatial coordinates, we see that for $\tau_k$ sufficiently small, $Q_{j,j+1}\le 0,$ since $\vv_k$ and $\f_k$ are assumed to be nonnegative. A virtually identical proof can be established for the off-diagonal entries affected by the Neumann boundary conditions. In addition, a similar calculation shows the lower diagonal of $Q$ is not positive. Therefore, $Q_{i,j}\le 0$ for $i\neq j.$

Next, we consider the row sums and show that these are nonnegative with at least one row sum being strictly positive. Consider the $i^{th}$ row sum of $Q$,
\bbbb
\sum_{j=1}^{M^2}Q_{i,j} &=& 1 - \frac{c_{21}\tau_k}{2\delta^2}\left((\uu_k)_{i-2}-2(\uu_k)_{i-1}+(\uu_k)_{i}\right)\\
&& ~~ - \frac{c_{21}\tau_k^2}{2}\sum_{j=1}^{M^2}\left[P \mbox{diag}\left((P+R)(d_1+s_1D_k^{(\uu)} + c_{12}D_k^{(\vv)})\uu_k+\f_k\right)\right]_{i,j},\eeee
for $i=1,\dots,M^2$. The boundary conditions are $(\uu_k)_{-1} = (\uu_k)_{1}$, $(\vv_k)_{-1} = (\vv_k)_{1}$, $(\uu_k)_{M^2+1} = (\uu_k)_{M^2-1}$, and $(\vv_k)_{M^2+1} = (\vv_k)_{M^2-1}$. Simplifying for $i\neq 1, M^2$ yields,
{\small
\bbbb
\sum_{j=1}^{M^2}Q_{i,j} & > &  \frac{1}{2} + \frac{c_{21}\tau_k}{\delta^2}(\uu_k)_{i-1} \\ && - \frac{c_{21}\tau_k^2}{2}\sum_{j=1}^{M^2}\left[P \mbox{diag}\left((P+R)(d_1+s_1D_k^{(\uu)} + c_{12}D_k^{(\vv)})\uu_k+\f_k\right)\right]_{i,j}\\
& = & \frac{1}{2} + \frac{c_{21}\tau_k}{\delta^2}\left[ (\uu_k)_{i-1} - \frac{\tau_k}{2} \mbox{diag}\left((P+R)(d_1+s_1D_k^{(\uu)} + c_{12}D_k^{(\vv)})\uu_k+\f_k\right)_{i-1,i-1}\right.\\
&& \left.~~~~~~~~ + \tau_k \mbox{diag}\left((P+R)(d_1+s_1D_k^{(\uu)} + c_{12}D_k^{(\vv)})\uu_k+\f_k\right)_{i,i}\right.\\
&& \left.~~~~~~~~ - \frac{\tau_k}{2}\mbox{diag}\left((P+R)(d_1+s_1D_k^{(\uu)} + c_{12}D_k^{(\vv)})\uu_k+\f_k\right)_{i+1,i+1}\right]\\
& = & \frac{1}{2} + \frac{c_{21}\tau_k}{\delta^2}(\uu_k)_{i-1}+ \OO{\left(\frac{\tau_k^2}{\delta^2}\right)}
%
\eeee}
Clearly, $\tau_k$ can be chosen sufficiently small such that the $\OO{\left(\tau_k^2/\delta^2\right)}$ term can be made small enough to ensure the result is positive. A similar result can be established for $i=1$ and $M^2$. Hence, we have that the weak row sum criterion is satisfied. Likewise, the remaining matrices can be shown to be nonsingular and inverse positive using a virtually identical approach.
\end{proof}

We now study the stability of both Eqs.~\R{ufactoredNOADI}-\R{ufactoredNOADI} and Eqs.~\R{sc1}-\R{sc2}. Moreover, we improve upon the previous methodology employed in \cite{BeauPadgett2017} by significantly reducing the previously imposed regularity conditions on the solution.
In order to prove stability of our proposed nonlinear splitting algorithm, we introduce a definition and some lemmas.

\begin{define}
Let $\|\cdot\|$ be an induced matrix norm. Then the associated logarithmic norm $\mu\,:\,\mathbb{C}^{M\times M}\to \RR$ of $A\in\mathbb{C}^{M\times M}$ is defined as
$$\mu(A) = \lim_{h\to 0^+}\frac{\|I + hA\| - 1}{h},$$
where $I\in\mathbb{C}^{M\times M}$ is the identity matrix.
\end{define}

\begin{rem}
When the induced matrix norm being considered is the $\ell^2-$norm, then $\mu(A) = \max\{\lambda\,:\,\lambda\ \mbox{is an eigenvalue of}\ (A+A^*)/2\}.$
\end{rem}



\begin{lemma}
Let $t\in\mathbb{C}$ and $A\in \mathbb{C}^{M\times M}.$ Then for any induced matrix norm $\|\cdot\|$ we have
$$\|E(tA)\| \le E(t\mu(A)),$$
where $E(\cdot)$ is the matrix exponential.
\end{lemma}

\begin{proof}
See \cite{Golub}.
\end{proof}

\begin{lemma}
Assume that \R{cfl} holds and let $L\mathrel{\mathop:}= P+R$, $L_u \mathrel{\mathop:}= d_1 + s_1 D_k^{(\emph{\uu})} + c_{12}D_k^{(\emph{\vv})}$, and $L_v \mathrel{\mathop:}= d_2 + s_2 D_k^{(\emph{\vv})} + c_{21}D_k^{(\emph{\uu})}$ Then we have
$$\l\|E\l(\tau_k L L_u)\r)\r\| \le 1 \quad\mbox{and}\quad \l\|E\l(\tau_k L L_v \r)\r\| \le 1.$$
\end{lemma}

\begin{proof}
By Lemma 4.1 we have
$$\l\|E\l(\tau_k L L_u \r)\r\| \le E\l(\tau_k\mu\l(L L_u\r)\r),$$
thus we need to show that $\mu\l(LL_u\r)\le 0$ to obtain the desired result. To do so, we need to show that the eigenvalues of the matrix $LL_u + L_u L^*$ are nonpositive.

Note that $L_u$ is diagonal and by \R{cfl} its entries are positive, thus we have $L L_uL_u^{-1/2} = L L_u^{1/2}.$ Then,
\begin{enumerate}
\item $LL_u$ is similar to $L_u^{1/2} L L_u^{1/2}$ with similarity matrix $L_u^{-1/2}$ and
\item $L$ is congruent to $L_u^{1/2} L L_u^{1/2}$ through the matrix $L_u^{-1/2}$.
\end{enumerate}
Therefore by Sylvester's law of inertia we have that the eigenvalues of $L$ are identical to those of $L_u^{1/2} L L_u^{1/2}$, which due to similarity, are the same as those of $LL_u$ \cite{Horn_MatrixAnalysis}.  Clearly, eigenvalues of $L$ are real and nonpositive \cite{Strikwerda}.  Hence, so are the eigenvalues of $LL_u$.  A similar argument gives that $LL_v$ has real and nonpositive eigenvalues. Subsequently, the desired bounds are established.
\end{proof}

\begin{rem}
It is worth noting that Lemma 4.2 holds independently of $k,$ and even $\uu_k$ and $\vv_k.$ All that is necessary is that the entries of the diagonal matrix $D$ are positive.
\end{rem}

\begin{lemma}
Assume that \R{cfl} holds and let $L$, $L_u$, and $L_v$ be defined as in the previous lemma. Then
$$\l\|E\l(\tau_k LL_u)\r) - E\l(\tau_k LL_v\r)\r\| \le C\tau_k\l\|LD_k^{(\emph{\uu}-\emph{\vv})}\r\|.$$
\end{lemma}

\begin{proof}
For convenience, we employ the notation $X_k^{(\uu)} \mathrel{\mathop:}= L L_u$ and $X_k^{(\vv)} \mathrel{\mathop:}= LL_v.$
By employing the inverse Laplace transform and the second resolvent identity \cite{Hille}, we have
\bbbb
&& E\l(\tau_k X_k^{(\uu)}\r) - E\l(\tau_k X_k^{(\vv)}\r) = \frac{1}{2\pi i}\int_{\Gamma} e^s\l[\l(sI-\tau_k X_k^{(\uu)}\r)^{-1} - \l(sI - \tau_k X_k^{(\vv)}\r)^{-1}\r]\,ds\nnn\\
&&~~~ = \frac{1}{2\pi i}\int_{\Gamma} e^s\l(sI - \tau_kX_k^{(\uu)}\r)^{-1}\tau_k\l(X_k^{(\uu)} - X_k^{(\vv)}\r)\l(sI - \tau_k X_k^{(\vv)}\r)^{-1}\,ds,
\eeee
where $\Gamma$ is a path surrounding the spectrum of the matrices $X_k^{(\uu)}$ and $X_k^{(\vv)}.$ Taking the norm of the above gives
\begin{eqnarray*}
&&\l\|E\l(\tau_k X_k^{(\uu)}\r) - E\l(\tau_k X_k^{(\vv)}\r)\r\| \le \\
&& ~~~\frac{\tau_k}{2\pi}\int_{\Gamma}|e^s|\l\|\l(sI - \tau_kX_k^{(\uu)}\r)^{-1}\r\|\l\|\l(X_k^{(\uu)} - X_k^{(\vv)}\r)\r\|\l\|\l(sI - \tau_k X_k^{(\vv)}\r)^{-1}\r\|\,|ds|.
\end{eqnarray*}
We proceed by developing separate bounds for the matrices in the above expression. Note that by the Laplace transform and Lemma 4.2, we have
\begin{eqnarray*}
\l\|\l(sI - \tau_kX^{(\uu)}\r)^{-1}\r\| &=& \l\|\int_0^\infty e^{-st}E\l(\tau_kX_k^{(\uu)}t\r)\,dt\r\| \\
&\le& \int_0^\infty |e^{-st}|\l\|E\l(\tau_kX_k^{(\uu)}t\r)\r\|\,|dt|\le |s^{-1}|.
\end{eqnarray*}
We similarly derive the bound
$$\l\|\l(sI - \tau_k X_k^{(\vv)}\r)^{-1}\r\| \le |s^{-1}|.$$
Further, we have by direct calculation
$$\l\|\l(X_k^{(\uu)} - X_k^{(\vv)}\r)\r\| \le \l\|LD_k^{(\uu)} - LD_k^{(\vv)}\r\| = \l\|LD_k^{(\uu-\vv)}\r\|.$$

Combining the above yields
$$
\l\|E\l(\tau_k X_k^{(\uu)}\r) - E\l(\tau_k X_k^{(\vv)}\r)\r\| \le \frac{\tau_k}{2\pi}\int_{\Gamma} |e^s|\l\|LD_k^{(\uu-\vv)}\r\||s^{-2}|\,|ds| \le C\tau_k\l\|LD_k^{(\uu-\vv)}\r\|,$$
which is the desired result.
\end{proof}





\begin{lemma}
Assume that \R{cfl} holds, $u,v\in H^1(\Omega),$ and that
$$\emph{\uu}_{k+1} = E\l(\tau_k L(d_1+s_1 D_{k+1/2}^{(\emph{\uu})} + c_{12} D_{k+1/2}^{(\emph{\vv})})\r)\emph{\uu}_{k} + \OO\l(\tau_k^3\r)$$
and
$$\emph{\vv}_{k+1} = E\l(\tau_k L(d_2+s_2D_{k+1/2}^{(\emph{\vv})}+c_{21}D_{k+1/2}^{(\emph{\uu})})\r)\emph{\vv}_{k} + \OO\l(\tau_k^3\r).$$
Then
$$\l\|LD_{k+1}^{(\emph{\uu}-\emph{\vv})}\r\| \le C$$
for some positive constant $C$.
\end{lemma}

\begin{proof}
By employing the notation from Lemma 4.3, we have
$$
\delta^{-2}\uu_{k+1} = (\delta^2X_k^{(\vv)})^{-1}\l[E\l(\tau_kX^{(\vv)}\r)\r]X_k^{(\vv)}\uu_k
$$
and
$$\delta^{-2}\vv_{k+1} = (\delta^2X_k^{(\uu)})^{-1}\l[E\l(\tau_kX^{(\uu)}\r)\r]X_k^{(\uu)}\vv_k.$$
Thus, it follows that
$$
\l\|LD_{k+1}^{(\uu-\vv)}\r\| \le \|\delta^2L\|\|\textbf{D}\|\le C_1\|\textbf{D}\|,$$
where
$$\textbf{D}\mathrel{\mathop:}= (\delta^2X_k^{(\vv)})^{-1}\l[E\l(\tau_kX^{(\vv)}\r)\r]X_k^{(\vv)}\uu_k - (\delta^2X_k^{(\uu)})^{-1}\l[E\l(\tau_kX^{(\uu)}\r)\r]X_k^{(\uu)}\vv_k.$$
Taking the norm gives
\bbb
\|\textbf{D}\| &\le & \l\|(\delta^2X_k^{(\vv)})^{-1}\r\|\l\|E\l(\tau_kX^{(\vv)}\r)\r\|\l\|X_k^{(\vv)}\uu_k\r\| \nnn \\
& & + \l\|(\delta^2X_k^{(\uu)})^{-1}\r\|\l\|E\l(\tau_kX^{(\uu)}\r)\r\|\l\|X_k^{(\uu)}\vv_k\r\|\nnn\\
& \le & C\l\|X_k^{(\vv)}\uu_k\r\| + C \l\|X_k^{(\uu)}\vv_k\r\|.\label{dd11}
\eee
By our assumptions and \cite{Ditz} we have
\bb{dd12}
\l\|X_k^{(\vv)}\uu_k\r\| \le C_1\quad\mbox{and}\quad \l\|X_k^{(\uu)}\vv_k\r\|\le C_2.
\ee
Applying \R{dd12} to \R{dd11} yields the desired result.
\end{proof}

In the following theorem we employ the following product notation for non-commuting matrices $A_j\in\mathbb{C}^{M\times M}:$
$$\prod_{j=m}^k A_j = \l\{\begin{array}{ll}
A_kA_{k-1}\cdots A_m, & \mbox{if}\ k\ge m,\\
I, & \mbox{if}\ k<m.
\end{array}\r.$$

\begin{thm}
Assume that \R{cfl} holds, $u,v\in H^1(\Omega),$ and let $\tau_j,\ j=0,\ldots,k,$ be sufficiently small. Then it follows that the schemes \R{ufactoredNOADI}-\R{ufactoredNNOADI} are stable in the sense that there exists $K_1,K_2>0$ such that
$$\l\|\emph{\z}_{k+1}^{(\emph{\uu})}\r\| \le K_1\quad\mbox{and}\quad \l\|\emph{\z}_{k+1}^{(\emph{\vv})}\r\|\le K_2,$$
where $\emph{\z}_0^{(\emph{\uu})} = \emph{\uu}_0 - \tilde{\emph{\uu}}_0$ and $\emph{\z}_0^{(\emph{\vv})} = \emph{\vv}_0 - \tilde{\emph{\vv}}_0$ are initial errors, $\emph{\z}_{k+1}^{(\emph{\uu})}\mathrel{\mathop:}= \emph{\uu}_{k+1}-\tilde{\emph{\uu}}_{k+1}$ and $\emph{\z}_{k+1}^{(\emph{\vv})}\mathrel{\mathop:}= \emph{\vv}_{k+1}-\tilde{\emph{\vv}}_{k+1}$ are the $(k+1)$th perturbed error vectors, and $K_1$ and $K_2$ are independent of $k$ and $\tau_k.$
\end{thm}

\begin{proof}
We first consider \R{ufactoredNOADI} and, for the time, disregard the nonlinear reaction term. Thus, we have
\bbb
\uu_{k+1} &=& \l(I-\frac{c_{12}\tau_k}{2}RD_{k+1}^{(\vv)}\r)^{-1}\l(I-\frac{c_{12}\tau_k}{2}PD_{k+1}^{(\vv)}\r)^{-1}\l(I-\frac{d_1\tau_k}{2}R\r)^{-1}\l(I-\frac{d_1\tau_k}{2}P\r)^{-1}\nnn\\
&&~~~\times \l(I+\frac{d_1\tau_k}{2}P\r)
\l(I+\frac{d_1\tau_k}{2}R\r)
\l(I+\frac{c_{12}\tau_k}{2}PD_{k}^{(\vv)}\r)
\l(I+\frac{c_{12}\tau_k}{2}RD_k^{(\vv)}\r)\uu_k\nnn\\
& \mathrel{\mathop:}= & M(\tau_k,\vv_k,\vv_{k+1})\uu_k.\label{mag1}
\eee
It will be convenient to appeal to a Magnus-type representation of \R{mag1} \cite{Iserles1,Iserles2}. That is, by employing the standard nonlinear Magnus-type integration and exponential splitting theory, we have
\bbb
&& M(\tau_k,\vv_k,\vv_{k+1}) = E\l(\frac{c_{12}\tau_k}{2}RD_{k+1}^{(\vv)}\r)E\l(\frac{c_{12}\tau_k}{2}PD_{k+1}^{(\vv)}\r)E\l(\frac{d_1\tau_k}{2}R\r)E\l(\frac{d_1\tau_k}{2}P\r)\nnn\\
&&~~\times E\l(\frac{d_1\tau_k}{2}P\r)E\l(\frac{d_1\tau_k}{2}R\r)E\l(\frac{c_{12}\tau_k}{2}PD_k^{(\vv)}\r)E\l(\frac{c_{12}\tau_k}{2}RD_k^{(\vv)}\r) + \OO\l(\tau_k^3\r)\nnn\\
&& =  E\l(\tau_k L(I+D_{k+1/2}^{(\vv)})\r) + \OO\l(\tau_k^3\r),\label{mag2}
\eee
where $L\mathrel{\mathop:}= P + R,$ $D_{k+1/2}^{(\vv)} = D(\vv_{k+1/2}),$ and $t_{k+1/2}\mathrel{\mathop:}=t_k+\tau_k/2.$ We now consider perturbations of the numerical solutions $\uu_k$ and $\vv_k,$ represented as
\bbb
\tilde{\uu}_{k+1} &=& E\l(\tau_k L(I+D_{k+1/2}^{(\tilde{\vv})})\r)\tilde{\uu}_k + \OO\l(\tau_k^3\r)\label{pert1}
\eee
and
\bbb
\tilde{\vv}_{k+1} &=& E\l(\tau_k L(I+D_{k+1/2}^{(\tilde{\uu})})\r)\tilde{\vv}_k + \OO\l(\tau_k^3\r)\label{pert2}
\eee
Let $\z^{(\uu)}_{k}\mathrel{\mathop:}= \uu_k - \tilde{\uu}_k$ and $\z^{(\vv)}_k\mathrel{\mathop:}= \vv_k - \tilde{\vv}_k.$ Iterating \R{mag1} by employing \R{mag2} yields
\bb{u1}
\uu_{k+1} = \l[\prod_{j=0}^k E\l(\tau_jL(I+D_{k+1/2}^{(\vv)}\r)\r]\uu_0 + \sum_{j=0}^k\OO\l(\tau_j^3\r)
\ee
and iterating \R{pert1} yields
\bb{uu1}
\tilde{\uu}_{k+1} = \l[\prod_{j=0}^k E\l(\tau_jL(I+D_{k+1/2}^{(\tilde{\vv)}}\r)\r]\tilde{\uu}_0 + \sum_{j=0}^k\OO\l(\tau_j^3\r).
\ee
Subtracting \R{uu1} from \R{u1} gives
\bbbb
\z_{k+1}^{(\uu)} = \Psi_1\z_0^{(\uu)} + \Psi_2\tilde{\uu}_0 + \sum_{j=0}^k\OO\l(\tau_k^3\r),
\eeee
where
$$\Psi_1\mathrel{\mathop:}= \prod_{j=0}^k E\l(\tau_jL(I+D_{j+1/2}^{(\vv)})\r)\quad\mbox{and}\quad\Psi_2\mathrel{\mathop:}= \prod_{j=0}^k F_j - \prod_{j=0}^k G_j,$$
with
$$F_j \mathrel{\mathop:}= E\l(\tau_jL(I+D_{j+1/2}^{(\vv)})\r)\quad\mbox{and}\quad G_j \mathrel{\mathop:}= E\l(\tau_jL(I+D_{j+1/2}^{(\tilde{\vv})})\r).$$

We first consider $\Psi_1.$ To that end, by Lemma 4.2, we have
$$\|\Psi_1\| \le \prod_{j=0}^k E\l(\tau_j\mu\l(L(I+D_{j+1/2}^{(\vv)})\r)\r) \le 1.$$
Now we consider $\Psi_2.$ By manipulating $\Psi_2,$ we arrive at the following expression
\bb{psi2}
\Psi_2 = \sum_{j=0}^k\l[\prod_{i=j+1}^{k}F_i\r]\l(F_j-G_j\r)\l[\prod_{i=0}^{j-1} G_{n+1-i}\r].
\ee
Taking the norm of both sides of \R{psi2} and employing Lemma 4.3 gives
\bb{bd3}
\|\Psi_2\| \le  \sum_{j=0}^k \l[\prod_{i=j-1}^{k} \|F_i\|\r] \|F_j-G_j\|\l[\prod_{i=0}^{j-1} \|G_{n+1-i}\|\r] \le 
\sum_{j=0}^k C\tau_j\l\|LD_{j+1/2}^{(\z^{(\vv)})}\r\|,
\ee
where $C$ is a positive constant. By our assumptions and Lemma 4.4, we have
\bb{bbdd}
\|\Psi_2\| \le \sum_{j=0}^k C\tau_j K \le C_1,
\ee
where $C_1$ is a constant independent of $\delta$ and $\tau_j,\ j=0,\ldots,k.$

Combining the above gives
\bbbb
\|\z_{k+1}^{(\uu)}\| & \le & \|\z_0^{(\uu)}\| + C_1\|\tilde{\uu}_0\| + \l\|\sum_{j=0}^k\OO\l(\tau_j^3\r)\r\|\nnn\\
& \le & \|\z_0^{(\uu)}\| + C_1\|\tilde{\uu}_0\| + C_2\max_{j=0,\ldots,k}\l\{\tau_j^2\r\}\sum_{j=0}^k \tau_j\\
& \le & \|\z_0^{(\uu)}\| + C_1\|\tilde{\uu}_0\| + C_3T \equiv K_1,\label{stab1}
\eeee
where $C_2$ and $C_3$ are independent of $j,\ \tau_j,$ $j=0,\ldots,k.$ By employing similar arguments with \R{ufactoredNNOADI}, it can be shown that $\|\z_{k+1}^{(\vv)}\| \le \|\z_0^{(\vv)}\| + C_1\|\tilde{\vv}_0\| + C_3T \equiv K_2,$ after possibly rescaling $C_1$ and $C_3.$

We now consider the nonlinear reaction term in the analysis of \R{ufactoredNOADI}. Recall that $\f_k = \f(\uu_k,\vv_k).$ Due to the differentiability of $f,$ we have
\bbbb
\f(\uu_k,\vv_k) - \f(\tilde{\uu}_k,\tilde{\vv}_k) & = & \f(\uu_k,\vv_k) - \f(\tilde{\uu}_k,\vv_k) + \f(\tilde{\uu}_k,\vv_k) - \f(\tilde{\uu}_k,\tilde{\vv}_k)\nnn\\
& = & \f_{\uu}(\xi_k,\vv_k)\z_k^{(\uu)} + \f_{\vv}(\tilde{\uu}_k,\eta_k)\z^{(\vv)}_k,
\eeee
for some $\xi_k = s_1\uu_k + (1-s_1)\tilde{\uu}_k,\ s_1\in[0,1]$ and $\eta_k = s_2\vv_k + (1-s_2)\tilde{\vv}_k,\ s_2\in[0,1].$ Thus, we have
\bbbb
\z_{k+1}^{(\uu)} = \Psi_1\z_0^{(\uu)} + \Psi_2\tilde{\uu}_0 + \sum_{j=0}^k \l[\frac{\tau_j}{2}\sum_{i=j}^{j+1}\l(\f_{\uu}(\xi_i,\vv_i)\z_i^{(\uu)} + \f_{\vv}(\tilde{\uu}_i,\eta_i)\z^{(\vv)}_i\r)\r] + \sum_{j=0}^k \OO\l(\tau_j^3\r).
\eeee
Since $f$ is bounded, we obtain a similar stability bound as above. The same can be done for \R{ufactoredNNOADI} by using the properties of the reaction term $g.$
\end{proof}



We now present the stability theorem for the scheme \R{sc1}-\R{sc2}.

\begin{thm}
Assume that \R{cfl} holds, $u,v\in H^1(\Omega),$ and let $\tau_j,\ j=0,\ldots,k,$ be sufficiently small. Then it follows that the schemes \R{sc1}-\R{sc2} are stable in the sense that there exists $K_1,K_2>0$ such that
$$\l\|\emph{\z}_{k+1}^{(\emph{\uu})}\r\| \le K_1\quad\mbox{and}\quad \l\|\emph{\z}_{k+1}^{(\emph{\vv})}\r\|\le K_2,$$
where $\emph{\z}_0^{(\emph{\uu})} = \emph{\uu}_0 - \tilde{\emph{\uu}}_0$ and $\emph{\z}_0^{(\emph{\vv})} = \emph{\vv}_0 - \tilde{\emph{\vv}}_0$ are initial errors, $\emph{\z}_{k+1}^{(\emph{\uu})}\mathrel{\mathop:}= \emph{\uu}_{k+1}-\tilde{\emph{\uu}}_{k+1}$ and $\emph{\z}_{k+1}^{(\emph{\vv})}\mathrel{\mathop:}= \emph{\vv}_{k+1}-\tilde{\emph{\vv}}_{k+1}$ are the $(k+1)$th perturbed error vectors, and $K_1$ and $K_2$ are independent of $k$ and $\tau_k.$
\end{thm}

\begin{proof}
This result follows by a proof similar to that of the previous theorem, where we use our developed bounds on $u$ and $v$ and employ the bound $s_1 D_k^{(\uu)} + c_{12} D_k^{(\vv)} \leq \max\{ s_1, c_{12}\} D_k^{(\textbf{w})},$ where $w_i = \max\{ \uu_i, \vv_i\}$ for $X_k^{(\uu)}$. The remaining arguments follow in a similar fashion.
\end{proof}

\section{Numerical Experiments}

In this section we provide illustrative the efficiency and convergence rate of the developed algorithm. 
All of the computations are carried out on a single HP EliteDesk 800 G1 work station with an Intel\circledR Core(TM) i7-4770 3.40 GHz processor with 16 GB of RAM using Matlab\circledR. 

\subsection{Example 1.}

Consider Eqs.~(\ref{eq:1.1})-(\ref{eq:1.2}) with $L=1,$ homogenous Dirichlet boundary conditions, and reaction functions
{\small
\begin{eqnarray*}
 f(u,v) = g(u,v) = -4\pi^2e^{-4\pi^2 t}\left( \cos^2(\pi y) \sin^2(\pi x) + \cos^2(\pi x) \sin^2(\pi y)-2\sin^2(\pi y) \sin^2(\pi x)\right) 
\end{eqnarray*}}
It can be shown by direct calculation that the exact solution to this system of partial differential equations is
$$ u(x,y,t) = v(x,y,t) = \sin(\pi x)\sin(\pi y) \exp(-2\pi^2 t). $$
Define $u_{i,j}^{(\tau)}$ and $v_{i,j}^{(\tau)}$ as the numerical solutions to $u$ and $v$ at time $T$ with a temporal step size of $\tau$ and location $(i\delta, j\delta)$. Based on the analysis is the previous section, we anticipate that the order of convergence in time is second-order.  That is, $\max\{ |u_{i,j}^{(\tau)}-u(ih,ij,T)|, |v_{i,j}^{(\tau)}-v(ih,ij,T)|\} \approx C \tau^p$ for some arbitrary constant $C$ and $p=2$.  We approximate $p$ by,
$$ p \approx \frac{1}{\ln(2)} \frac{1}{N^2} \max\left\{ \sum_{i,j=1}^N \ln \frac{ \left|u_{i,j}^{(\tau)}-u(ih,ij,T)\right|}{\left|u_{i,j}^{(\tau/2)}-u(ih,ij,T)\right|}, \sum_{i,j=1}^N \ln \frac{ \left|v_{i,j}^{(\tau)}-v(ih,ij,T)\right|}{\left|v_{i,j}^{(\tau/2)}-v(ih,ij,T)\right|}\right\} $$
Let $\delta=.01$ and $d_1=d_2=s_1=s_2=c_{12}=c_{21}=1$. Therefore, based on our CFL condition, we let $\tau = .25\times 10^{-4}<.5\times 10^{-4}.$  We simulate the solution until $T=1$ and estimate the approximate order of convergence to be $p \approx 2.000809.$  This indicates that our nonlinear operator splitting method is converging at the desired second-order rate.

Likewise, if we consider Neumann boundary conditions with the reaction functions
\begin{eqnarray*}
&& f(u,v)= g(u,v) = \exp(-2\pi^2t)\pi^2(9\exp(\pi^2t)\cos(\pi x)\cos(\pi y)+8\cos(\pi x)^2\cos(\pi y)^2\\
 &&~~~~~~~~~~~~~~~~~~~~~~~~~~~-4\cos(\pi y)^2\sin(\pi x)^2-4\cos(\pi x)^2\sin(\pi y)^2),
\end{eqnarray*}
then the exact solution for $u$ and $v$ is $a+\cos(\pi x)\cos(\pi y)\exp(-\pi^2t)$, where $a\in\mathbb{R}$.  In the simulations we let $a=1$, which ensures that the initial condition remains nonnegative throughout the computational domain. Using the same resolution as in the Dirichlet case, we determine the order of convergence to be approximately $1.9945$. In Table \ref{Table1}, we demonstrate the data for the maximum absolute error as we increase the resolution in time.  Notice that for a fixed $\delta,$ as we half the temporal step the max error decreases by approximately one-fourth. This is also true if we fix the temporal step size and then halve the spatial step sizes, which is the expected result for a second-order numerical method.

\begin{table}
\begin{center}
\begin{tabular}{|c|c|c|c|}\hline
$\tau$ & $\delta$ & Max Error (Dirichlet) & Max Error (Neumman) \\ \hline
   $2.5000\times 10^{-5}$ &  $1.0000\times 10^{-2}$ &  $1.8804\times 10^{-8}$ & $4.4342\times 10^{-8}$ \\ \hline
   $1.2500\times 10^{-5}$ &  $1.0000\times 10^{-2}$ &  $4.6915\times 10^{-9}$ & $1.0839\times 10^{-9}$\\ \hline
   $6.2500\times 10^{-6}$ &  $1.0000\times 10^{-2}$ &  $1.1694\times 10^{-9}$ & $2.4808\times 10^{-9}$\\ \hline
\end{tabular}
\end{center}
\caption{A absolute maximum error at $t=1$ as we increase the resolution of the temporal step. Each time that the temporal step is halved we see that the max error is reduced by a fourth, which indicates a second-order convergence of the operator splitting scheme.}
\label{Table1}
\end{table}

To test the computational efficiency we determine the computational time to complete $1000$ iterations for a fixed $\tau$ as we increase the spatial resolution, $N$. We determined that the computational time scales as $\OO(N^{1.7568}),$ which is an exponent less than the number of unknowns. Figure 1(b) shows the log-log plot of the computational time versus $N$. These results suggests that the procedure is computationally efficient.

\begin{figure}[h]
\begin{center}
\includegraphics[scale=.4]{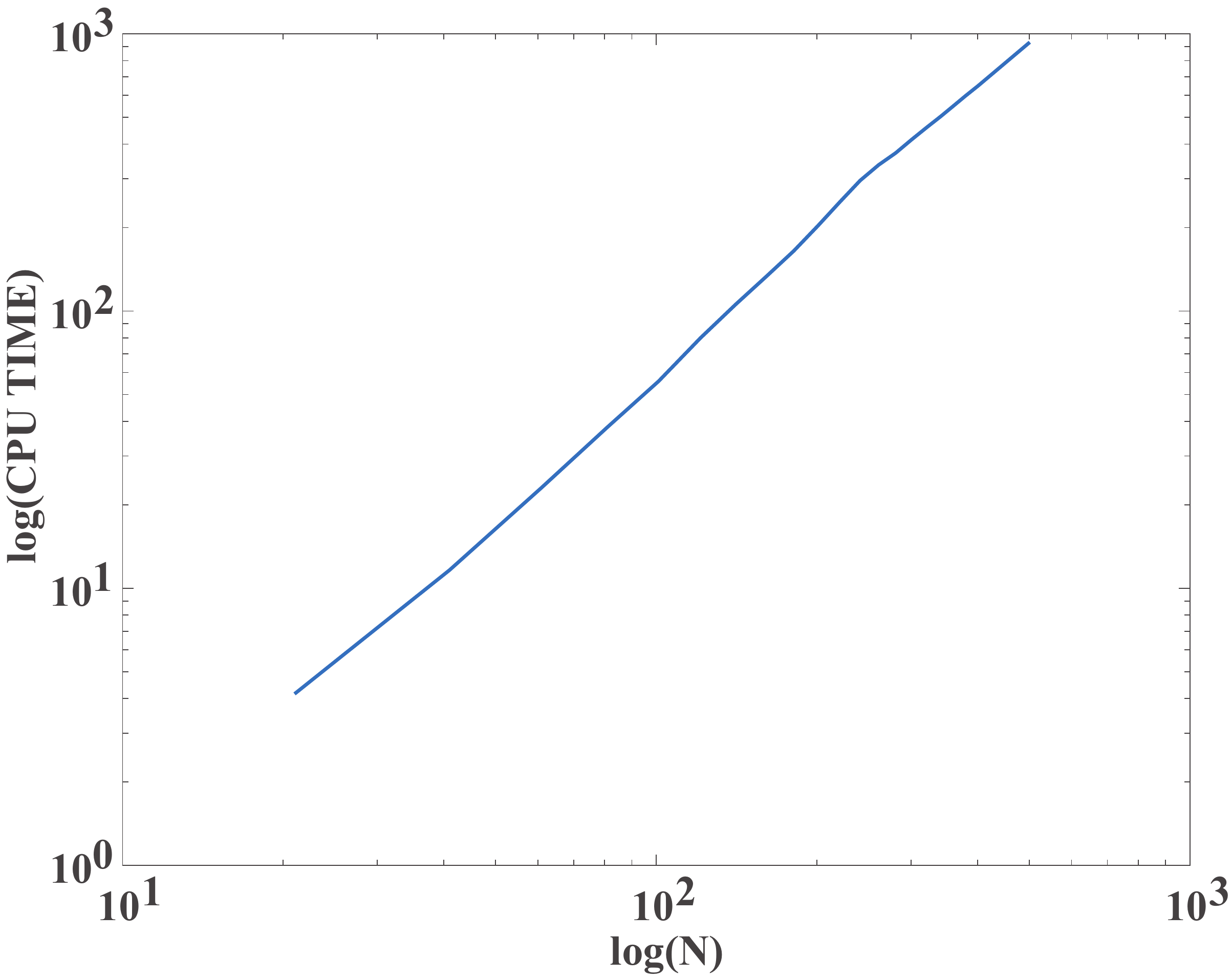}
\caption{A log-log plot of the computational time, in seconds, versus $N$ after $1000$ iterations.  The temporal step is held constant, $\tau = 10^{-6}$, while $\delta=1/(N-1)$.  A linear least squares approximates the slope of the line to be $1.75681$. This indicates that the computational time is proportional to $N^{1.654628}$.  Since this is slower than $N^2$, then the proposed nonlinear splitting scheme is highly efficient. }
\label{Example1CompTime}
\end{center}
\end{figure}

\subsection{Example 2.}

Consider Eqs.~(\ref{eq:1.1})-(\ref{eq:1.2}) with Neumann boundary conditions and the following Lotka-Volterra reaction functions,
\begin{eqnarray*}
 f(u,v) &=& u(a_1 - b_1 u + c_1 v) \\
 g(u,v) &=& v(a_2 + b_2 u - c_2 v)
\end{eqnarray*}
In \cite{Le2006} it was shown that for the above Lotka-Volterra reaction function, a global solution exists if the parameters for diffusion satisfy the following conditions:

\[ s_1s_2 \geq 0, ~~ s_2 > c_{12}, ~~ s_1 > c_{21} \]

Consider a situation that violates the third condition and let $d_1=.01, d_2=.1, s_1=.05, s_2=.4, c_{12}=.12,$ and $c_{21}=.06.$ These coefficients satisfy the criteria shown in \cite{Kouachi2014} for global solutions. For the reaction functions' parameters we choose $a_1=1,b_1=2,c_1=.2,a_2=.3,b_2=1,$ and $c_2=4.$  The initial conditions are chosen to be
\begin{eqnarray*}
u(x,y,0) &=& 2+\sum_{i=1}^3 \sigma_i \cos(n_i x)\cos(m_i y) \\
v(x,y,0) &=& 2+\sum_{i=1}^3 \beta_i \cos(a_i x)\cos(b_i y),
\end{eqnarray*}
where $a_i,b_i,m_i,$ and $n_i$ are positive integers and $L=\pi.$  The amplitudes $\sigma_i$ and $\beta_i$ were chosen randomly from uniform distribution on the interval $(0,1)$. Regardless of the choice of amplitudes or frequencies the solution converges to spatially homogenous solutions which exists globally in time.

\subsection{Example 3.}

Consider Eqs.~(\ref{eq:1.1})-(\ref{eq:1.2}) with Dirichlet boundary conditions and reaction functions of the form
\begin{eqnarray*}
 f(u,v) &=& u(a_1 + b_1 u) \\
 g(u,v) &=& v(a_2 + b_2 v).
\end{eqnarray*}
In \cite{Zhu2014} it was shown that finite time blow up can occur if $b_i>s_i\lambda$ and $a_i\geq d_i \lambda$ where $\lambda$ is the first eigenvalue of the Laplacian operator and $c_{12}=c_{21}=0$.

Consider the initial conditions
\begin{eqnarray*}
u(x,y,0) = v(x,y,0) = \sin^2(4x)\sin^2(2y)
\end{eqnarray*}
with parameters $d_1=d_2=1, s_1=s_2=.05, a_1=a_2=3, b_1=b_2 = 4$. In such a situation the system of equations are decoupled. The parameters for $u$ and $v$ are chosen to obey the criteria to guarantee blow-up. In order to satisfy Eq.~(\ref{cfl}) the temporal step-size must be reduced as the populations grow in time.  We terminate the calculation when $\tau$ is adapted to a minimum of $10^{-10}$. Indeed, in our simulations the populations are demonstrated to grow without bound. Figure \ref{SelfBlowUpPanel} shows a panel of the initial condition coupled with two snapshots of the population during the calculation.  Figure \ref{SelfBlowUpPanel}(b) shows the population at $t=.5$ for which the population has begin to concentrate and grow in the center of the domain.  Figure \ref{SelfBlowUpPanel}(c) shows the population in the iteration prior to termination.  At termination of the computation, the population had reached a maximum value of approximately $1.6370\times 10^7$, while the temporal derivative was approximately $1.0116\times 10^{18}$.

\begin{figure}[h]
\begin{center}
\includegraphics[scale=.35]{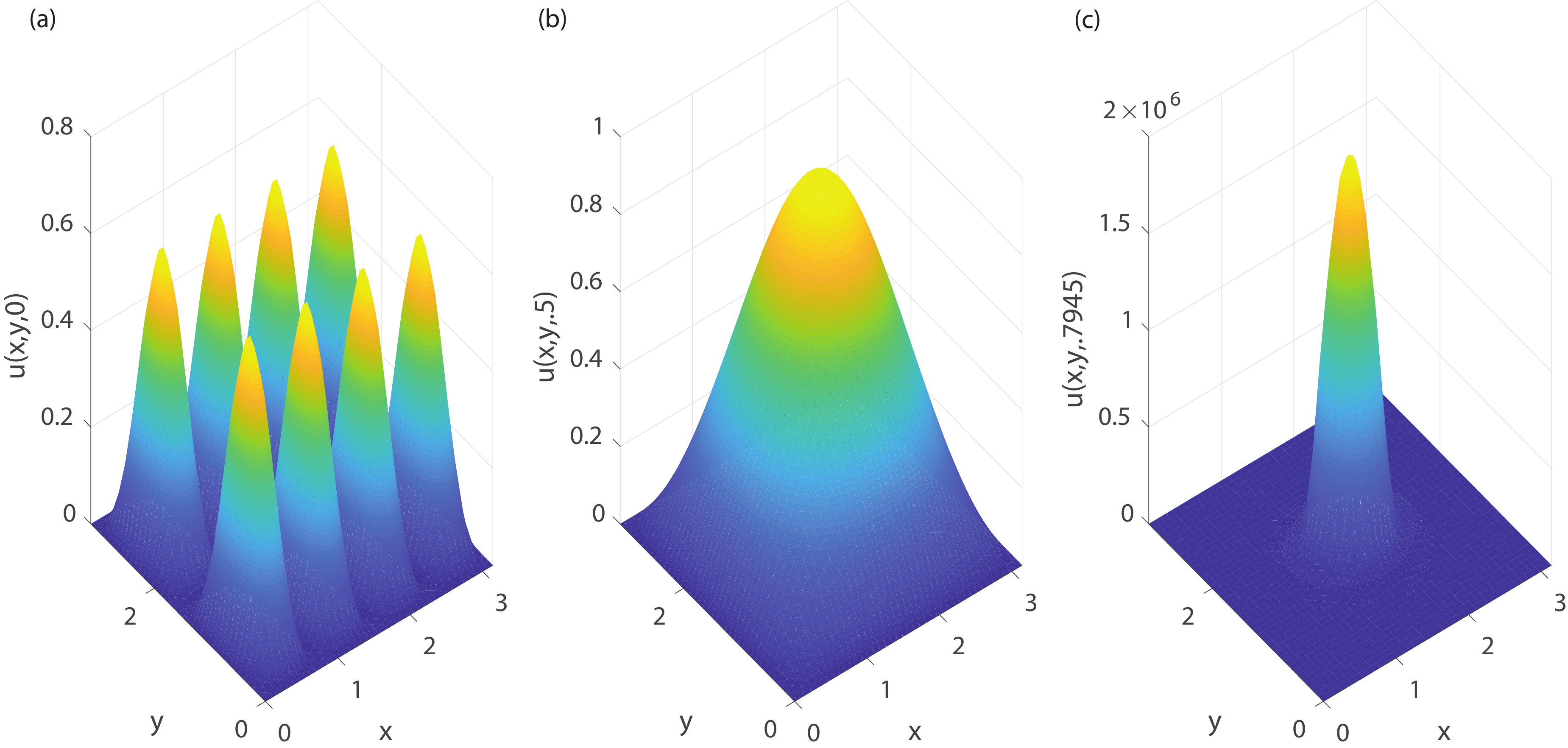}
\caption{A panel showing the (a) initial condition, (b) population at $t=.5$ and (c) just prior to finite time blow-up at $t=.7945$, from left to right, respectively.}
\label{SelfBlowUpPanel}
\end{center}
\end{figure}

\section{Conclusions}

In this paper a novel adaptive nonlinear operator splitting method was developed to approximate coupled parabolic partial differential equations that contain self- and cross-diffusion terms. The method extends the algorithm designed in \cite{BeauPadgett2017} and provides an updated criterion to ensure the method's stability and the invertibility of the involved matrices. Surprisingly, the inclusion of the cross-diffusion term does not complicate the numerical analysis any further than the inclusion of self-diffusion terms. Hence, convergence and stability follow by employing techniques similar to those developed in our prior work \cite{BeauPadgett2017}. However, the current research improves upon these techniques and employs techniques that allow for reduced regularity requirements on the solution. Moreover, the analysis does not depend on the boundary conditions utilized and can be readily extended to several problems of particular interest in the mathematical biology community such as in \cite{Kouachi2014} and more recently \cite{Parshad2016}. Our numerical experiments provide empirical evidence of the anticipated convergence rate for Dirichlet and Neumann boundary conditions, further verifying the efficacy of the proposed method.  The computational experiments further provide evidence suggesting that global solutions exist under certain parameter restrictions.

\section*{Acknowledgements}

The first author would like to explicitly express his gratitude for an internal research grant (No. 150030-26423-150) from Stephen F. Austin State University.

~

\noindent We would like to express our gratitude for the numerous suggestions provided by the reviewers that undoubtedly improved the results and readability of this manuscript.

%

\end{document}